\documentclass[11pt]{amsart}
\usepackage{amsmath, amsthm, amssymb}
\usepackage{graphicx}
\usepackage{mathbbol}
\usepackage{txfonts}
\usepackage[usenames]{color}

\title[Almost Rigidity of the AH Positive Mass Theorem]{Almost Rigidity of the Positive Mass Theorem \\
for Asymptotically Hyperbolic Manifolds \\ with Spherical Symmetry}

\author[Sakovich]{Anna Sakovich}
\address{Anna Sakovich: Department of Mathematics, Uppsala University}
\email{anna.sakovich@math.uu.se}

\author[Sormani]{Christina Sormani}
\address{Christina Sormani: Department of Mathematics, Lehman College and CUNY Graduate Center}
\email{sormanic@gmail.com}

\thanks{The authors began this research while in residence at the Mathematical Sciences Research Institute (MSRI) funded by NSF Grant No. 0932078000. The research was also funded in part by Sormani's NSF grant DMS-1612409.
}

\theoremstyle{plain}

\theoremstyle{definition}

\newtheorem{thm}{Theorem}[section]
\newcommand{\bt}{\begin{thm}}
\newcommand{\et}{\end{thm}}

\newtheorem{cor}[thm]{Corollary}   

\newcommand{\bc}{\begin{cor}}
\newcommand{\ec}{\end{cor}}

\newtheorem{lem}[thm]{Lemma}   

\newcommand{\bl}{\begin{lem}}
\newcommand{\el}{\end{lem}}

\newtheorem{prop}[thm]{Proposition}
\newcommand{\bp}{\begin{prop}}
\newcommand{\ep}{\end{prop}}

\newtheorem{defn}[thm]{Definition}
\newtheorem{conj}[thm]{Conjecture}

\newcommand{\bd}{\begin{defn}}    
\newcommand{\ed}{\end{defn}}

\newtheorem{rmrk}[thm]{Remark}   

\newcommand{\br}{\begin{rmrk}}
\newcommand{\er}{\end{rmrk}}

\newcommand{\grad}{\nabla}

\newcommand{\be}{\begin{equation}}

\newcommand{\ee}{\end{equation}}

\newcommand{\N}{\mathbb{N}}

\newcommand{\R}{\mathbb{R}}

\newcommand{\diam}{\operatorname{diam}}

\newcommand{\vol}{\operatorname{Vol}}

\DeclareMathOperator{\RS}{SphSym}

\DeclareMathOperator{\arcsinh}{arcsinh}

\begin{document}

\maketitle

\begin{abstract}
We use the notion of intrinsic flat distance to address the 
almost rigidity of the positive mass theorem for asymptotically hyperbolic manifolds. In particular, we prove that a sequence of spherically symmetric asymptotically hyperbolic manifolds satisfying the conditions of the positive mass theorem converges to hyperbolic space in the intrinsic flat sense, if the limit of the mass along the sequence is zero.
\end{abstract}

\section{{\bf{Introduction}}}\label{sec:intro}

Broadly speaking, an asymptotically hyperbolic manifold 
is a non-compact Riemannian manifold having an end such 
that its  geometry approaches that of hyperbolic space.  Such 
manifolds naturally arise when studying isolated systems in
General Relativity.  When one assumes the dominant energy
condition they are endowed with
scalar curvature greater than or equal to $-m(m-1)$.  

The notion of mass 
for this class of manifolds was first defined mathematically by Wang in \cite{Wang} and Chru\'sciel-Herzlich in \cite{Chrusciel-Herzlich}.   Earlier
exploration of a possible notion of mass was completed by physicists Abbott-Deser, Ashtekar-Magnon and Gibbons-Hawking-Horowitz-Perry \cite{Abbott-Deser}\cite{Ashtekar-Magnon}\cite{Gibbons-Hawking-Horowitz-Perry}.  For the family of anti-de Sitter-Schwarzschild metrics (also called Kottler metrics) the mass coincides with the mass parameter. 
This notion of mass is similar to the 
concept of ADM mass of asymptotically Euclidean 
manifolds in the sense that it is a coordinate invariant which 
is computed in a fixed asymptotically hyperbolic end and which measures the leading order deviation of the geometry from the background hyperbolic metric in the end. 

Recall that the Euclidean Positive Mass Theorem, proven by Schoen and Yau \cite{Schoen-Yau-PMT} \cite{Schoen-Yau-NEW}, states that asymptotically Euclidean manifolds with nonnegative scalar curvature have nonnegative mass and the mass is $0$ iff the manifold is isometric to Euclidean space. The Hyperbolic Positive Mass Theorem states that an asymptotically hyperbolic manifold with scalar curvature $R \ge -m(m-1)$ has nonnegative mass and the mass is $0$ iff the manifold is isometric to hyperbolic space.  Wang and Chru\'sciel-Herzlich prove this theorem in the Spin setting following techniques similar to Witten's proof of the Euclidean Positive Mass Theorem \cite{Chrusciel-Herzlich}\cite{Wang}\cite{Witten-positive-mass}.
Andersson-Cai-Galloway prove it without spin assumption but with assumptions on the dimension and on the geometry at infinity \cite{Andersson-Cai-Galloway}.  
The first author has proven the Hyperbolic Positive Mass Theorem for general asymptotics in dimensions $m=3$ using an adaptation of Schoen and Yau's Jang equation reduction technique \cite{Schoen-Yau-PMT-2}.  With the rigidity statement of the positive mass theorem in mind, it is natural to ask
the following question concerning the almost rigidity of the statement.

{\em{ What happens if the mass of an asymptotically hyperbolic manifold is close to zero and the scalar curvature is at least $-m(m-1)$. Must the manifold then be close to hyperbolic space in some appropriate sense? 
}}

The same question of stability has been addressed in relation to the rigidity part of the Euclidean Positive Mass Theorem: {\em an asymptotically Euclidean manifold with small mass and non-negative scalar curvature should be close to Euclidean space in some appropriate sense}.  Asymptotically Euclidean spin manifolds with small mass have been studied by Finster with Bray and Kath, see \cite{Bray-Finster}\cite{Finster-Kath}\cite{Finster}. From estimates on the spinor field in Witten's argument they find that the $L^2$-norm of the curvature tensor (over the manifold minus an exceptional set) is bounded in terms of the mass. Lee \cite{Lee-near-equality} studies asymptotically Euclidean manifolds which are conformally flat outside a compact set $K$. For such manifolds he proves that the conformal factor can be controlled by the mass, so that the conformal factor tends uniformly to one outside any ball containing $K$ as the mass tends to zero. The argument by Lee relies on the variational argument used by Schoen and Yau in \cite{Schoen-Yau-PMT} to prove the rigidity part of the positive mass theorem for asymptotically Euclidean manifolds, and does not require the manifold to be spin. 

Asymptotically hyperbolic manifolds have rich geometry at infinity which often prevents the application of techniques designed to deal with asymptotically Euclidean manifolds. For instance, it is not clear whether an argument similar to that used by Bray, Finster and Kath 
in \cite{Bray-Finster}\cite{Finster-Kath}\cite{Finster} 
could work to prove stability of the positive mass theorem for asymptotically hyperbolic manifolds. The difficulty comes from the fact that the spinor field in the proof of positive mass theorem in the work of Chru\'sciel-Herzlich and Wang has more complicated behavior near infinity than the spinor field in Witten's original argument \cite{Chrusciel-Herzlich}\cite{Wang}\cite{Witten-positive-mass}.
Still, the argument of Lee in \cite{Lee-near-equality} can be adapted to the asymptotically hyperbolic setting, although this adaptation is not straightforward, see the
work of Dahl, Gicquaud and the first author in \cite{Dahl-Gicquaud-Sakovich} for further details. Using this result, one may draw conclusions on convergence to hyperbolic space outside some compact set. 

The current paper focuses on the much harder problem of understanding what happens inside the compact set.   When studying the almost rigidity of a
positive mass theorem in either setting one serious concern is that even if the mass is small, there still can be arbitrarily deep ``gravity wells'' and "bubbles".   One avoids bubbles by restricting to the class of manifolds with no closed interior minimal surfaces.  However, even in the spherically symmetric setting, Lee and the second author have constructed sequences of manifolds which develop increasingly deep gravity wells which do not converge in the smooth, Lipschitz or Gromov-Hausdorff sense to Euclidean space \cite{Lee-Sormani}.  However they prove intrinsic flat convergence of the sequences.  The intrinsic flat distance between Riemannian manifolds measures the filling volume between the spaces and will be described within.  It was first introduced by the second author and Wenger in \cite{SorWen2} applying sophisticated ideas of Ambrosio-Kirchheim \cite{AK} extending earlier work of Federer-Fleming \cite{FF} and Whitney \cite{Whitney}.

Lee and the second author proposed the Intrinsic Flat Almost Rigidity of the Euclidean Positive Mass Theorem Conjecture:  {\em a sequence of asymptotically flat manifolds with nonnegative scalar curvature and
no closed interior minimal surfaces whose ADM mass converges to $0$
converges in the pointed intrinsic flat sense to Euclidean space.}   
They stated and then proved this conjecture in the spherically symmetric setting \cite{Lee-Sormani}. It was proven in the graph setting by Huang, Lee, and the second author in \cite{Huang-Lee-Sormani} applying prior work of Huang-Lee \cite{Huang-Lee}.  It was proven in the geometrostatic case by the second author with Stavrov in \cite{Sormani-Stavrov}.  See also related work \cite{Lee-Sormani-2} of Lee and the second author, where the almost rigidity of the Penrose inequality for spherically symmetric asymptotically Euclidean manifolds was studied, and a recent work \cite{Allen} of Allen which addresses stability of positive mass theorem and Penrose inequality with respect to $L^2$ metric convergence under the assumption that a smooth uniformly controlled solution of the inverse mean curvature flow exists.  

Here we propose the Intrinsic Flat Almost Rigidity of the Hyperbolic Positive Mass Conjecture:   {\em a sequence of asymptotically hyperbolic manifolds with $R \ge -m(m-1)$ and no closed interior minimal surfaces whose mass converges to $0$
converges in the pointed intrinsic flat sense to hyperbolic space.}  To make this conjecture more precise we must identify where the points are located as we cannot allow them to diverge to infinity or descend into a gravity well.  Thus we imitate one of the precise conjectures in \cite{Lee-Sormani}, locating the points on constant mean curvature surfaces although it is possible we could locate them on other canonically defined surfaces. 

\begin{conj}
Let $M_j^m$ be a sequence of asymptotically hyperbolic Riemannian manifolds
with $R \ge -m(m-1)$ whose mass approaches $0$.    Assume that $M_j^m$
contains no closed interior minimal surfaces and that either $M_j^m$ has no boundary or the boundary is a closed minimal surface.  Fix $A_0>0$ and
take any $D>0$.  Suppose that $\Sigma_j \subset M_j$  
are constant mean curvature (CMC) surfaces of fixed area $A_0$. 
Then the intrinsic flat distance between the tubular neighborhoods converges to 0:
\be
\lim_{j\to \infty} d_{\mathcal{F}}\left(\,T_D(\Sigma_j)\subset M_j^m\,,\, T_D(\Sigma_\infty)\subset \mathbb{H}^m\,\right)\, = 0
 \ee
 and
 \be
 \lim_{j\to \infty} \vol\left(T_D(\Sigma_j)\subset M_j^m\right) =
 \vol\left(T_D(\Sigma_\infty)\subset \mathbb{H}^m\right).
 \ee
Here $\Sigma_\infty$ is a sphere with $\vol_{m-1}(\Sigma_\infty)=A_0$ about any
 point $p_\infty$
 in hyperbolic space, $\mathbb{H}^m$, and $T_D(\Sigma)$ is the 
tubular neighborhood of radius $D$ around $\Sigma$.    
Thus if $p_j \in \Sigma_j$ then $(M_j,p_j)$ converges in the pointed intrinsic 
flat sense to $({\mathbb{H}}^m, p_\infty)$.
\end{conj}

Note that for a sufficiently large given area the choice of a constant mean curvature (CMC) surface with that given area is unique.  There is in fact a canonical CMC foliation of asymptotically hyperbolic manifolds with $R \ge -m(m-1)$ just as Huisken-Yau and others proved in the asymptotically Euclidean case  \cite{Huisken-Yau}.  Indeed in 2004, Rigger used Huisken-Yau's method to prove the existence and uniqueness of such a foliation for manifolds asymptotic to the spatial anti-de Sitter-Schwarzschild solution \cite{Rig04}. This result was generalized by Neves-Tian and Chodosh, \cite{NT09}\cite{NT10}\cite{Cho14}.  Cederbaum, Cortier, and the first author proved that the CMC-foliation characterizes the center of mass in the hyperbolic setting \cite{CCS}.   Their work has been further generalized by Nerz in \cite{Nerz} which surveys the prior work in this direction thoroughly.  

Here we test the validity of our new conjecture by proving it in the spherically symmetric case (see Theorem~\ref{thm-main} and
Definition~\ref{def-rot-sym} within).  This class of spaces includes anti-de Sitter-Schwarzschild manifolds (see Remark \ref{Kottler}) and the asymptotically hyperbolic spherically symmetric
gravity wells:

\begin{defn}\label{def-rot-sym}
Given $m\ge 3$,
let $\RS_m$ be the class of complete $m$-dimensional 
$SO(m)$-spherically symmetric smooth Riemannian manifolds with scalar curvature $\mathrm{R}\geq -m(m-1)$ which have no closed interior
minimal hypersurfaces and either have no boundary or have a boundary
which is a stable minimal hypersurface.  
These spaces are warped products with
metrics as in (\ref{eqn-warp-metric}) satisfying (\ref{eqn-f'>0}) and
(\ref{eqn-18}). 
\end{defn}

Our main result is the following theorem.

\begin{thm}\label{thm-main}   
Given any $\epsilon>0$, $D>0$, $A_0>0$ and $m\in \N$,
there exists a $\delta=\delta(\epsilon, D, A_0, m)>0$
such that if $M^m\in\RS_m$ has mass, 
$\mathrm{m}_{\mathrm{AH}}(M)<\delta$, and 
$\mathbb{H}^m$ is hyperbolic space of the same dimension,
then
\be
 d_{\mathcal{F}}(\,T_D(\Sigma_0)\subset M^m\,,\, T_D(\Sigma_0)\subset \mathbb{H}^m\,)\,<\, 
 \epsilon,
 \ee
 and 
 \be
 |\vol(T_D(\Sigma_0)\subset M_j^m) -
 \vol(T_D(\Sigma_0)\subset \mathbb{H}^m)| <  \epsilon 
 \ee
 where  $\Sigma_0$ is the symmetric sphere of area  
$\vol_{m-1}(\Sigma_0)=A_0$, and $T_D(\Sigma_0)$ is the 
tubular neighborhood of radius $D$ around $\Sigma_0$.   
Precise estimates on $\delta=\delta(\epsilon, A_0, D, m)$ are provided 
within the proof. 
\end{thm}

The paper is organized as follows. In Section \ref{Sect-Embed} we review
the key results and theorems about intrinsic flat convergence needed to
prove our theorem. In Section \ref{Sect-Positive} we summarize the properties of the spherically symmetric manifolds in $\RS_m$, provide the background material on the positive mass theorem and discuss some immediate consequences of having small mass  in the spherically symmetric setting. These results are applied in Section \ref{Sect-Pos}, where we prove Theorem \ref{thm-main}. 

\vspace{.3cm}

\noindent{\bf{Acknowledgments:}}   The authors would like to thank the organizers of the Mathematical General Relativity Program at MSRI and especially James Isenberg, Piotr Chru\'sciel, Greg Galloway, and Richard Schoen for encouraging everyone to work together so actively.  We would also like to thank H\'el\`ene Barcelo and David Eisenbud for everything they do to make MSRI such an amazing place work. The time we spent there was invaluable to our careers.  We would also like to thank 
Hubert Bray, Otis Chodosh, Greg Galloway, and Pengzi Miao for
organizing the AMS Special Session at Charleston where we completed this first paper in our international collaboration.   We hope to have many more delightful
interactions with everyone for years to come.

\section{{\bf{Estimating the Intrinsic Flat Distance}}}
\label{Sect-Embed}

In this section we review the definition of the intrinsic flat distance and the
theorems we will apply to prove our main theorem.  The notion was defined by the second author and Wenger in \cite{SorWen2} as a distance, $d_{\mathcal{F}}(M^m_1, M^m_2)$, between 
oriented compact Riemannian manifolds with boundary, $M^m_j$, and other more general spaces called integral current spaces.  
It was proven to satisfy symmetry, the triangle inequality, and
most importantly,  $d_{\mathcal{F}}(M^m_1, M^m_2)=0$ if and only if there is
an orientation preserving isometry between the spaces.  In order to rigorously define the intrinsic flat distance and integral current spaces, and prove this theorem one needs the work of  Ambrosio-Kirchheim \cite{AK} which builds upon the
Geometric Measure Theory developed by Federer-Fleming \cite{FF}.   

\subsection{Filling Manifolds and Excess Boundaries}

In this paper we need only estimate the intrinsic flat distance between Riemannian manifolds with boundary and so we do not need the full definition which requires
Geometric Measure Theory.  Instead we apply the following:
\be \label{eqn-def-intrinsic-flat-1}
d_{\mathcal{F}}(M^m_1, M^m_2) \le \inf_{A, B \subset Z}\left\{\vol_{m+1}(B^{m+1}) +\vol_m(A^m)\right\}
\ee
where we are taking an infimum over all piecewise smooth
Riemannian manifolds, $Z^{m+1}$, 
such that there are distance preserving embeddings, $\psi_i: M_i^m \to Z^{m+1}$,
and over all oriented submanifolds $B^{m+1}\subset Z^{m+1}$ and 
$A^m \subset Z^{m+1}$
such that
\be \label{eqn-def-intrinsic-flat-2}
\int_{\psi_1(M_1)}\omega -\int_{\psi_2(M_2)}\omega=\int_{\partial B}\omega + \int_A\omega 
\ee
for any differential $m$-form $\omega$ on $Z^{m+1}$.   In this case, 
$B^{m+1}$ is called a {\em filling manifold} between $M_1$ and $M_2$
and $A^m$ is called the {\em excess boundary} and each might have
more than one component.

 \subsection{Distance Preserving Embeddings vs Riemannian Isometric Embeddings}
 
 Note that a distance preserving embedding, 
$\psi: M \to Z$ is a map such that
\be\label{eqn-def-metric-isom-embed}
d_Z(\psi(x), \psi(y)) = d_M(x,y) \qquad \forall x,y \in M.
\ee
This is significantly stronger than a Riemannian isometric embedding
which preserves only the Riemannian structure and thus lengths
of curves but not distances between points as in 
(\ref{eqn-def-metric-isom-embed}).   For example 
$\psi: {\mathbb{S}}^1 \to {\mathbb{S}}^2$ mapping the circle to the
equator of a sphere is both a distance preserving embedding
and a Riemannian isometric embedding.
However $\psi: {\mathbb{S}}^1 \to {\mathbb{D}}^2$,
 mapping a circle to the boundary of a flat disk 
is a Riemannian isometric embedding but not a distance preserving embedding.

If one were to use
Riemannian isometric embeddings in the definition of intrinsic flat distance
then any manifold would lie arbitrarily close to any other manifold, as one
could create a warped product through a point between them of arbitrarily
small volume.  For example, take $\psi: {\mathbb{S}}^1 \to Z^2$ mapping
to the $\epsilon$ level set of $Z^2= [0, \epsilon] \times_f {\mathbb{S}}^1$ 
where $f$ increases from $f(0)=0$ to $f(\epsilon)=1$.  This is a Riemannian
isometric embedding and the $\vol(Z^2)< \epsilon$.  The reader should be warned that distance preserving embeddings as in (\ref{eqn-def-metric-isom-embed})
are called "isometric embeddings" by metric geometers like Gromov, Ambrosio,
Kirchheim and the second author in some of her papers.  
 
 In many settings it is far easier to construct a Riemannian isometric embedding than it is to find a distance preserving embedding.  In fact we will be constructing
 Riemannian isometric embeddings in this paper from the regions in our
 spherically symmetric asymptotically hyperbolic manifolds into 
 ${\mathbb{H}}^m \times \mathbb{R}$ with the standard product metric.  Once we have constructed those Riemannanian isometric embeddings we will need to construct filling spaces $Z$
 and distance preserving embeddings of our regions into $Z$.  The following collection of theorems proven by  Lee and the second author  in \cite{Lee-Sormani} will enable us to complete these constructions.  
 
 We begin with a definition which quantifies how far a
 Riemannian isometric embedding is from a distance preserving embedding:
 
\begin{defn}
Let $\varphi: M \to N$ be a Riemannian isometric embedding.  Then the
embedding constant may be defined:
\be \label{eqn-embed-const-1}
C_M:= \sup_{p,q\in M} \left( d_M(p, q) - d_N(\varphi(p),\varphi(q)) \right).
\ee
\end{defn}

In \cite{Lee-Sormani} Lee and the second author prove the following theorem:

\begin{thm}\label{thm-Z}
Let $M^m$ be a compact Riemannian manifold with boundary
defined by the graph
\be
M^m=\{(x,z):\, z=F(x), \, x\in W\} \subset W \times \R
\ee
where $F: W\to \R$ is differentiable and $W$ is
a Riemannian manifold with boundary.
Viewed
as a Riemannian isometric embedding into $W\times \R$,
the embedding constant satisfies
\be \label{eqn-C_M-bound-1}
C_M\le 2\diam(W) \sup\left\{|\grad F_x|: \, x\in W\right\}. 
\ee
\end{thm}

Note that a better estimate for $C_M$ is available, see \cite[Remark 3.7]{Lee-Sormani}. However, the estimate \eqref{eqn-C_M-bound-1} suffices for the purposes of the current paper.   

\subsection{Constructing Filling Manifolds}\label{ss-constr}
 
 We now present a few theorems from the work of Lee and the second author
 \cite{Lee-Sormani} which start with a pair of manifolds that have Riemannian
 isometric embeddings into a common manifold, and then construct a new
 space $Z$ into which they have distance preserving embeddings.  These
 theorems allow us to estimate the intrinsic flat distances between these manifolds.

We begin with the basic construction:

\begin{thm}\label{thm-embed-const}
Let $\varphi: M \to N$ be a Riemannian isometric embedding and let
\be \label{eqn-embed-const-1}
C_M:= \sup_{p,q\in M} \left( d_M(p, q) - d_N(\varphi(p),\varphi(q)) \right).
\ee
Taking
\be\label{eqn-embed-const-S}
S_M= \sqrt{C_M(\diam(M)+C_M)},
\ee
we may construct a piecewise smooth Riemannian manifold, $Z$,
by attaching a strip of width $S_M$ to the manifold $N$ as follows:
\be
Z=\left\{(x,s): \, x\in \varphi(M),\, s\in [0, S_M]\right\} \cup \left\{(x,0): \,x\in N\right\}  
\subset N \times [0, S_M]
\ee
so that we have a distance preserving embedding,
$\psi: M \to Z$, from $M$ into the far side of the strip given by 
\be
\psi(x)=(\varphi(x), S_M).
\ee   
Here the distances
in $Z$, denoted $d_Z$, are found by taking the infimum of lengths of
curves in $Z$ just as in any Riemannian manifold.
\end{thm}

As a consequence of Theorem \ref{thm-embed-const}, the following bound on the intrinsic flat distance was obtained in \cite{Lee-Sormani}.   We start with two
manifolds isometrically embedded into a third, and then
$Z$ is constructed by applying Theorem~\ref{thm-embed-const} twice, building a strip for each $M_i$.  The intrinsic flat distance between the $M_i$ is then
the sum of the volumes of the strips and their boundaries and the volume between the Riemannian isometric embeddings and the last piece of excess boundary.

\begin{thm}\label{embed-const}
If $\varphi_i: M^m_i \to N^{m+1}$ are Riemannian isometric embeddings with 
embedding constants $C_{M_i}$ as in (\ref{eqn-embed-const-1}), and if
they are disjoint and lie in the boundary of a region $W \subset N$
then 
\begin{eqnarray}
d_{\mathcal{F}}(M_1, M_2) &\le& 
S_{M_1}\left(\vol_m(M_1)+ \vol_{m-1}(\partial M_1) \right) \\
&&+S_{M_2}\left(\vol_m(M_2)+ \vol_{m-1}(\partial M_2) \right)\\
&&+ \vol_{m+1}(W) + \vol_{m}(V)
\end{eqnarray}
where $V= \partial W \setminus \left( \varphi_1(M_1) \cup \varphi_2(M_2)\right)$, and $S_{M_i}$ are defined in (\ref{eqn-embed-const-S}).
\end{thm}

In a more general case when a pair of manifolds does not have global
Riemannian
isometric embeddings into a common manifold $N^{m+1}$ 
the following result from \cite{Lee-Sormani} can be applied.   Here we have
only regions $U_i \subset M_i$ which embed into a common $N^{m+1}$
so the regions are used to construct the strips and the extra parts of
$M_i \setminus U_i$ are just extra pieces of excess boundary.

\begin{thm}\label{embed-const-2}
If $M^m_i$ are Riemannian manifolds and $U^m_i\subset M^m_i$
are submanifolds that have Riemannian
isometric embeddings $\varphi_i: U^m_i \to N^{m+1}$ with 
embedding constants $C_{U_i}$ as in (\ref{eqn-embed-const-1}), and if
their images are disjoint and lie in the boundary of a region $W \subset N$
then
\begin{eqnarray}
d_{\mathcal{F}}(M_1, M_2) &\le& 
S_{U_1}\left(\vol_m(U_1)+ \vol_{m-1}(\partial U_1) \right)\\
&&+S_{U_2}\left(\vol_m(U_2)+ \vol_{m-1}(\partial U_2) \right)\\
&&+ \vol_{m+1}(W) + \vol_{m}(V) \\
&&+\vol_{m}(M_1\setminus U_1)
+\vol_m(M_2\setminus U_2)
\end{eqnarray}
where 
$V= \partial W \setminus \left( \varphi_1(U_1) \cup \varphi_2(U_2)\right)$
where $S_{U_i}$ are defined in (\ref{eqn-embed-const-S}).
\end{thm}


It should be noted that there are many other ways of estimating intrinsic
flat distances and indeed such methods are applied in papers
exploring the almost rigidity of the positive mass theorem in settings 
without rotational symmetry.

\section{\bf{Scalar Curvature and Mass in Asymptotically Hyperbolic Setting}}
\label{Sect-Positive}

This section is organized as follows. In the first subsection we briefly review the properties of spherically symmetric manifolds and  
the key formulas defining their mass.  In the next subsection we embed a manifold in this class into $\mathbb{H}^m\times \mathbb{R}$ as a graph and review the positive mass theorem and 
the monotonicity of the Hawking mass.  In the third
subsection we explore geometric implications of
having a small mass and prove key lemmas
which will be applied later for proving the stability of the
positive mass theorem.  We are deliberately imitating the methods
used by Lee and the second author in \cite{Lee-Sormani} to make this
easier to read for those who know that work.

\subsection{Understanding Mass in our Setting}\label{ss-setting}

We consider a spherically symmetric manifold, $(M^m,g)\in \RS_m$,
as described in Definition~\ref{def-rot-sym}. We can write its metric in
geodesic coordinates as 
\be \label{eqn-warp-metric}
g=ds^2 + f(s)^2 g_0
\ee
where 
$f:[0,\infty) \to [0,\infty)$
and where $g_0$ is the standard metric on the $(m-1)$-sphere.
Let $f_{min}=f(0)$.   

When $\partial M=\emptyset$ then $f(0)=f_{min}=0$ and 
$f(s)\ge f(0)$ by smoothness at the pole, $p_0$, where $s$ is
the distance from the pole.  When $\partial M \neq \emptyset$ the definition of
$\RS$ states that  $\partial M$ is a stable
minimal surface so $f(0)=f_{min}>0$ and $f'(0)=0$ and $f(s)\ge f(0)$
in that case as well.  In this case $s$ is the distance from $\partial M$. 

Let $\Sigma'_s$ be a level set of this distance function at a distance
$s$ from the pole or boundary.
We then have the following formulae for the ``area" and mean curvature  
of $\Sigma'_s$:
\be \label{A-s}
\mathrm{A}(s) =\vol_{m-1}(\Sigma'_s)=\omega_{m-1} f^{m-1}(s),
\ee
\be\label{H-s}
\mathrm{H}(s)  = \frac{(m-1)f'(s)}{f(s)}.
\ee
Thus $\Sigma'_s$ provide a constant mean curvature foliation of the manifold.   

The definition of $\RS$
also requires that $M^m$ has no interior minimal surfaces,
so by (\ref{H-s}), we have
\be
f'(s)\neq 0 \qquad \forall s\in (0,\infty].  
\ee
By the mean value theorem, we see that
\be\label{eqn-f'>0}
f'(s)>0 \qquad \forall s\in (0,\infty].  
\ee
Thus $A(s)$ is increasing and we can uniquely define our
rotationally symmetric constant mean curvature spheres
\be
\Sigma_{\alpha_0}=\Sigma'_{s_0} \textrm{ such that } 
\vol_{m-1}(\Sigma'_{s_0})=\alpha_0.
\ee

Recall that if $\Sigma$ is a connected surface of spherical topology in a three dimensional asymptotically hyperbolic manifold  then 
its Hawking mass may be defined as 
\be
\mathrm{m}_{\mathrm{H}}(\Sigma) = \frac{1}{2} \sqrt{\frac{A}{\omega_2}} 
\left(1 - \frac{1}{4\pi} \int_\Sigma \left(\left(\frac{H}{2}\right)^{2}-1\right) d\sigma \right),
\ee 
see e.g. \cite{Gibbons}.   Thus
\be
\mathrm{m}_{\mathrm{H}}(\Sigma'_s) =
 \frac{f(s)}{2}\left(1-(f'(s))^2 + f^2(s)\right).
 \ee

We extend this to define a higher dimensional Hawking mass function, 
$\mathrm{m}_{\mathrm{H}}(s)$, for $M^m\in \RS$:
\be \label{eqn-hawking-1}
\mathrm{m}_{\mathrm{H}}(s) = 
\left(\frac{f^{m-2}(s)}{2}\right)\left(1-(f'(s))^2 + f^2(s)\right). 
\ee
We obtain the monotonicity of the Hawking Mass a la Geroch,
\be \label{eqn-hawking-increases}
\mathrm{m}_{\mathrm{H}}'(s) = \left(\frac{f^{m-1}(s)f'(s)}{2(m-1)}\right) (\mathrm{R}+m(m-1))\ge 0
\ee
because $f'(s)>0$ for $s\in (0,\infty)$
and the scalar curvature at any point in $\Sigma'_s$ satisfies
\be \label{eqn-18}
\mathrm{R} =
\frac{(m-1)\left( (m-2)(1- (f'(s))^2) - 2f(s)f''(s)\right)} {f^2(s)}\geq -m(m-1).
\ee
Observe that when
$\partial M \neq \emptyset$, we have
\be \label{eqn-rmin-1}
\mathrm{m}_{\mathrm{H}}(0)=\tfrac{1}{2}(f_{min}^{m-2}+f^m_{min})>0.
\ee 
When $\partial M =\emptyset$, we have $\mathrm{m}_{\mathrm{H}}(0)=0$.

We define the mass of $M^m$ as the limit of the Hawking masses:
\be \label{eqn-limit}
\mathrm{m}_{\mathrm{AH}}(M^m) =\lim_{s\to \infty} \mathrm{m}_{\mathrm{H}}(s) \in [0,\infty].
\ee
Thus we have
\be \label{eqn-Hawking-inequality}
0\le\mathrm{m}_{\mathrm{H}}(0) \le \mathrm{m}_{\mathrm{H}}(s) \le \mathrm{m}_{\mathrm{AH}}.
\ee

\begin{rmrk}
A computation shows that for this limit to be finite it is necessary that 
\be
f(s) = \sinh s (1 + O(e^{-ms})).
\ee
 As a consequence of these asymptotics and spherical symmetry, one can check that the limit in \eqref{eqn-limit} coincides with the notion of mass by Wang \cite{Wang} and Chru\'sciel and Herzlich \cite{Chrusciel-Herzlich}. More specifically, this limit equals the first component of the mass vector (see e.g. \cite[Definition 3.1]{Herzlich}), whereas the remaining components vanish, so the Minkowskian length of the mass vector is precisely \eqref{eqn-limit}. In other words, in the setting of the current paper \eqref{eqn-limit} agrees with the standard definition of mass in arbitrary dimensions.   See also Remark 2.9 on the equivalence of these different definitions of mass in \cite{Nerz}.
\end{rmrk}

\subsection{Finding a Riemannian Isometric Embedding}
\label{ss-Riemannian}

Since our manifolds are spherically symmetric and asymptotically
hyperbolic we can embed them into $\mathbb{H}^m \times \mathbb{R}$ with the standard product metric as follows:
                                 
\begin{lem}\label{lem-graph}
Given $M^m \in \RS_m$, we can find a spherically
symmetric Riemannian isometric embedding of $M^m$
into $\mathbb{H}^m \times \mathbb{R}$
as the graph of some radial function $z=z(r)$ satisfying $z'(r)\geq 0$.  In graphical coordinates, we have 
\be \label{eqn-graph-g}
g=(1+[z'(r)]^2)dr^2 + \sinh^2 r \, g_0,
\ee
with $r\ge r_{min} = \arcsinh f_{min}$. 
Furthermore, we have the following formulae for
the scalar curvature, 
area, mean curvature, Hawking mass and its derivative
in terms of the radial coordinate $r$:
\begin{alignat}{1}
\mathrm{R} (r) 
&= \frac{m-1}{1+(z')^2}\left( -m + \frac{(m-2)(z')^2}{\sinh^2 r} + \frac{z' z'' \sinh 2r }{(1+(z')^2) \sinh^2 r}\right),\\
A (r)& =\omega_{m-1} \sinh^{m-1} r, \\
H (r)& = \frac{(m-1)\coth r}{\sqrt{1+(z')^2}},\\
\mathrm{m}_{\mathrm{H}} (r)& = \frac{\sinh^{m-2} r \cosh^2 r}{2}\left(\frac{(z')^2}{1+(z')^2}\right), \\
\mathrm{m}_{\mathrm{H}}' (r)& = \frac{\sinh^{m-1} r \cosh r}{2(m-1)}  \left(\mathrm{R} + m(m-1)\right).
\end{alignat}
The isometric embedding is unique up to a choice of $z_{min}=z(r_{min})$.
\end{lem}

\begin{proof}
Since Hawking mass is nonnegative (\ref{eqn-hawking-1}) and
there are no closed interior minimal surfaces (\ref{eqn-f'>0}), we have 
\be
0 < f'(s) \leq \sqrt{1 + f^2(s)}< \infty.
\ee   
We first define our radial function
\be
r(s)= \arcsinh (f(s)),
\ee
so $r_{min}= \arcsinh (f_{min})$ and
\be
\frac{dr}{ds}=\frac{f'(s)}{\sqrt{1 + f^2(s)}} \leq 1.
\ee 
Since $s$ is a distance function, we need to define $z(r)$ so that
\be
s'(r)=\sqrt{1+(z'(r))^2}.
\ee
This is solvable because $s'(r) \geq 1$.   We choose
\be
z'(r) = \sqrt{(s'(r))^2 -1} \ge 0,
\ee
and so we choose an arbitrary $z_{min}$ and obtain
\be
z(r)= \int_{r_{min}}^r \sqrt{(s'(r))^2 -1} \, dr + z_{min}.
\ee
The formulae for scalar curvature, 
area, mean curvature, Hawking mass and its derivative in $r$ follow from the corresponding equations in $s$ 
(see Section \ref{ss-setting}).
\end{proof}

\subsection{Reviewing Rigidity when the Mass is Zero}

We now review the proof of the Hyperbolic Positive Mass Theorem and 
Penrose Inequality in the spherically symmetric setting, as this helps 
convey what to expect when proving our main theorem.   It is the quantitative nature of the proof of these theorems in this setting that makes it possible to prove our conjecture.   There is no quantitative proof of the Hyperbolic
Positive Mass Theorem without such strong symmetry.  The key here is the
monotonicity of Hawking mass:

\begin{thm}\label{thm-monotonicity}
Given $M^m \in \RS_m$ isometrically
embedded into $\mathbb{H}^m \times \mathbb{R}$ as above,
we have
\be
\mathrm{m}_{\mathrm{H}}(r_{min})\le \mathrm{m}_{\mathrm{H}}(r)\le \mathrm{m}_{\mathrm{AH}}
\ee
and if there is an equality then $M^m$ is hyperbolic space (when $\mathrm{m}_{\mathrm{AH}}=0$) or a Riemannian anti-de Sitter-Schwarzschild manifold of mass $\mathrm{m}_{\mathrm{AH}}>0$,
\be \label{eqn-def-Sch}
g=\left(1+\frac{-2 \mathrm{m}_{\mathrm{AH}}}{2 \mathrm{m}_{\mathrm{AH}} -\sinh^{m-2}r \cosh^2 r}\right)dr^2 + \sinh^2 r\, g_0.
\ee
\end{thm}

\begin{rmrk}\label{Kottler}
If we set $\rho = \sinh r$ in (\ref{eqn-def-Sch})
we recover the more standard form 
of the metric of a Riemannian anti-de Sitter-Schwarzschild manifold:
\be
g=\frac{d\rho^2}{1 + \rho^2 - 2\mathrm{m}_{\mathrm{AH}} \rho^{2-m}} + \rho^2 \,g_0.
\ee
\end{rmrk}

\begin{proof}
The monotonicity of the Hawking mass follows from 
(\ref{eqn-hawking-increases}).   When there is an equality
we apply Lemma~\ref{lem-graph} to see that
\be
\mathrm{m}_{\mathrm{AH}} = \frac{\sinh^{m-2}r \cosh^2 r}{2}\left(\frac{(z')^2}{1+(z')^2}\right). 
\ee
It follows that 
\be \label{eqn-Sch}
(z')^2=\frac{-2 \mathrm{m}_{\mathrm{AH}}}{2 \mathrm{m}_{\mathrm{AH}} -\sinh^{m-2}r \cosh^2 r}, 
\ee
which implies \eqref{eqn-def-Sch}.

When $\mathrm{m}_{\mathrm{AH}}=0$, we have $z'(r)=0$ so $z \equiv z_{min}$ and $M^m$ is the hyperbolic space. Observe that in this case $r_{min}$ must be $0$ because $r_{min}>0$
forces the existence of a minimal surface at the boundary, and coordinate spheres are not minimal in hyperbolic space.
\end{proof}

Note that in the above proof we not only proved that when the mass is 0 our manifold is isometric to Hyperbolic space, but in fact it embeds naturally with $z'(r)=0$.   When the mass is almost $0$ one might hope that we would have
$z'(r)$ close to $0$ but that is not true.  There can be a deep well or
horizon where $z'(r)$ is arbitrarily large for small values of $r$.   Examples of
this sort are constructed explicitly in \cite{Lee-Sormani} and could easily be
done in this setting as well.

\subsection{Preparing for Deep Apparent Horizons}

The following lemma will be useful for dealing with 
``deep apparent horizons'' 
that may occur in sequences satisfying the hypothesis of
Theorem~\ref{thm-main}:

\begin{lem}\label{lem-radial} 
When $\partial M \neq \emptyset$,
we have
\be
g=(1+[r'(z)]^2) dz^2 + \sinh^2 r(z) \, g_0,
\ee  
and the following formulae for scalar curvature,
area, mean curvature, Hawking mass, and derivative of the Hawking mass in terms of the height coordinate $z$ hold on $M$:
\begin{alignat}{1}
\mathrm{R}(z) &= -m(m-1) + \frac{m-1}{(1+(r')^2)\sinh^2 r}\left( m \cosh^2 r - 2 - \frac{r'' \sinh 2r}{1+(r')^2}\right), \\
A (z)& =\omega_{m-1} \sinh^{m-1} r, \\
H (z)& = \frac{(m-1)r' \coth r }{\sqrt{1+(r')^2}},\\
\mathrm{m}_{\mathrm{H}} (z)& = \frac{\sinh^{m-2} r \cosh^2 r}{2(1 + (r')^2 )}, \\
\mathrm{m}_{\mathrm{H}}' (z)& = \frac{r' \sinh^{m-1} r \cosh r}{2(m-1)} \left(\mathrm{R} + m(m-1)\right).
\end{alignat}
When $\partial M = \emptyset$ these formulae hold for $r\geq r_{disk}$ for some $r_{disk} \in [0,\infty)$, unless $(M,g)$ 
is hyperbolic space.
\end{lem}

There is a similar lemma in \cite{Lee-Sormani} but here things are a little
more complicated as they are throughout.

\begin{proof}
It is clear from the proof of the Lemma~\ref{lem-graph} that $z'(r) > 0$ unless $\mathrm{m}_{\mathrm{H}}(r) = 0$. 
If $\partial M \neq \emptyset$ then by \eqref{eqn-rmin-1} and monotonicity of Hawking mass we have $\mathrm{m}_{\mathrm{H}}(r)>0$, and hence 
$z'(r) > 0$, for all $r \geq r_{min}$. Replacing $dr=(z'(r))^{-1} dz$ yields the stated formula for $g$, and the rest 
of the formulae follow from the respective equations of Lemma \ref{lem-graph}. 

Assume now that $\partial M =  \emptyset$. If $\mathrm{m}_{\mathrm{H}}(r)>0$ for $r>0$ the same argument applies, and the formulae hold on $M$. 
Otherwise, if the Hawking mass is not strictly positive for $r>0$, define $r_{disk}$ by 
\be
r_{disk}=\sup\{r: \, \mathrm{m}_{\mathrm{H}}(r) =0\} \in (0,\infty].
\ee
If $r_{disk}\neq \infty$ then all the equations will hold for $r > r_{disk}$. If $r_{disk} = \infty$, 
then by \eqref{eqn-limit} we have $\mathrm{m}_{\mathrm{AH}}=0$. Applying Theorem \ref{thm-monotonicity} we conclude
that $(M,g)$ is hyperbolic space.
\end{proof}

\subsection{Bounding the Diameter of the Boundary}
\label{ss-bounding}

In this subsection we use the mass $\mathrm{m}_{\mathrm{AH}}$ to bound the diameter of the boundary of the manifold:

\begin{lem} \label{lem-rmin}
If $M^m\in \RS$ then
\be
r_{min}\le r_\infty
\ee
where $r_\infty$ is the unique nonnegative root of the equation
\be
2 \mathrm{m}_{\mathrm{AH}} =  \sinh^{m-2}r_\infty \cosh^2 r_\infty .
\ee
So $\diam(\partial M^m) \le \pi \sinh r_\infty$.
\end{lem}

\begin{proof}
Note that when $r_{min}=0$ the inequality $r_{min}\le r_\infty$ holds trivially. Assuming $r_{min}>0$, we know by Lemma~\ref{lem-radial}
with $z_{min}=z(r_{min})$ that
\be
0 = \frac{(m-1)r'(z_{min})\coth r_{min}}{\sqrt{1+(r'(z_{min}))^2}}\\
\ee
because the boundary is a minimal surface. So
$r'(z_{min})=0$ and the Hawking mass is
\be
\mathrm{m}_{\mathrm{H}}(z_{min})  = \tfrac{1}{2} \sinh^{m-2}r_{min}\cosh^2 r_{min}.
\ee
By the monotonicity of Hawking mass (\ref{eqn-hawking-increases}) we have 
$\mathrm{m}_{\mathrm{H}}(z_{min}) \leq \mathrm{m}_{\mathrm{AH}}$, so if $r_\infty$ is defined as above then 
\be
\sinh^{m-2}r_{min} \cosh^2 r_{min} \leq  \sinh^{m-2}r_\infty \cosh ^m r_\infty,
\ee 
and $r_{min}\le r_\infty$ follows, since the function $h(r)=\sinh^{m-2}r \cosh^2 r$ is strictly increasing.
\end{proof}

\subsection{Lipschitz Control Away from the Center}

Here we see that if we avoid small values of $r$, where there might be a deep
well or a deep horizon, we obtain Lipschitz controls on $z=F(r)$
where we have embedded as in Lemma~\ref{lem-graph}.   These
controls will allow us to estimate the embedding constants of annular
regions in our embedded manifolds.

\begin{lem}\label{lem-F'}
If $z=F(r)$ then
\be
F'(r) \ge \sqrt{\frac{2m_1}{\sinh^{m-2}r \cosh^2 r - 2 m_1 }\,} \qquad \forall r > r_1
\ee
for any $r_1 \ge r_{min}$, where $m_1 = \mathrm{m}_{\mathrm{H}}(r_1)$, and
\be
F'(r) \le \sqrt{\frac{2\mathrm{m}_{\mathrm{AH}}}{\sinh^{m-2}r \cosh^2 r-2\mathrm{m}_{\mathrm{AH}}}\, } 
\qquad \forall r \ge \max \left\{ r_1, r_\infty \right\}
\ee
where $\mathrm{m}_{\mathrm{AH}}=\mathrm{m}_{\mathrm{AH}}(M^m)$, and $r_\infty$ is as in Lemma~\ref{lem-rmin}. 
\end{lem}

\begin{proof}
By the monotonicity of Hawking mass (\ref{eqn-hawking-increases})
for $r>r_1$ we have
\be
 \mathrm{m}_{\mathrm{H}}(r_1)\leq \mathrm{m}_{\mathrm{H}}(r) \leq  \mathrm{m}_{\mathrm{AH}}, 
\ee
hence
\be
m_1 \le  \frac{\sinh^{m-2}r \cosh^2 r}{2}
 \left(\frac{(z')^2}{1+(z')^2} \right) \leq  \mathrm{m}_{\mathrm{AH}}
\ee
by Lemma~\ref{lem-graph}. 
Solving this inequality for $z'$ yields
\be
z' \geq \sqrt{\frac{2m_1}{\sinh^{m-2}r \cosh^2 r-2m_1}} \qquad \forall r>r_1
\ee
and
\be
z' \leq \sqrt{\frac{2\mathrm{m}_{\mathrm{AH}}}{\sinh^{m-2}r \cosh^2 r-2\mathrm{m}_{\mathrm{AH}}}} \qquad \forall r\ge r_\infty.
\ee
\end{proof}


\section{\bf{Almost Rigidity of the Positive Mass Theorem}}
\label{Sect-Pos}

In this section we will prove Theorem~\ref{thm-main}
following a method very similar to that used in \cite{Lee-Sormani}, 
which can be briefly described as follows.
We wish to estimate the intrinsic flat distance between the
two tubular neighborhoods, $T_D(\Sigma_{\alpha_0})\subset M^m$
and $T_D(\Sigma_{\alpha_0})\subset {\mathbb{H}}^m$ by
constructing an explicit filling between them.
By Lemma ~\ref{lem-graph}, both $M^m$
and ${\mathbb{H}}^m$ have Riemannian isometric embeddings
into $N^{m+1}={\mathbb{H}}^{m}\times {\mathbb{R}}$.  In fact the
isometric embedding of ${\mathbb{H}}^m$ is distance preserving.
We will then
use Theorem~\ref{thm-embed-const} to construct
a filling space $Z$ from $N^{m+1}$ by adding a strip for $M^m$.  
This will allow to explicitly
estimate the intrinsic flat distance by adding up the appropriate volumes
of the strips and the regions between the embeddings in $N^{m+1}$.  

The difficulty arises when we note that we only have a small embedding
constant for annular regions in our manifold that avoid possible deep wells
and deep horizons.  Without a small embedding constant, the volumes
in the strips will be too large.  Thus we will cut off the deep wells at a carefully
chosen radius $r_\epsilon$, and only embed the annular regions beyond
$r_\epsilon$ into $N^{m+1}$.    We will then estimate the intrinsic
flat distance using Theorem~\ref{embed-const-2}. 

\subsection{Constructing Z}

We will use the same notation that was used by Lee and the second author in \cite{Lee-Sormani} to label
all the relevant regions that will be used in the construction of the
filling space.  See Figure~\ref{fig-SS-flat-filling}.  These are regions whose volumes will later need to be estimated
to compute the intrinsic flat distance.

Recall that given $\alpha>0$ we denote by  $\Sigma_{\alpha}$ the constant mean curvature sphere with $\vol_{m-1} \Sigma_{\alpha} = \alpha$ and we define its ``area radius'' $r(\Sigma_{\alpha})$ by
\be
r(\Sigma_{\alpha}) = \sinh^{-1}( (\alpha/\omega_{m-1})^{1/(m-1)} ).
\ee

Now fix $\alpha_0>0$ and set the following notations for the various key radii:
\begin{eqnarray}
r_{min}&=& \inf\{r(p): \, p\in M^m\},\\
r_{D-}&=&\inf \{ r(p): p\in T_D(\Sigma_{\alpha_0}) \subset M^m\},\\
r_0&=& r(\Sigma_{\alpha_0})= \sinh^{-1}( (\alpha_0/\omega_{m-1})^{1/(m-1)} ),\\
r_{D+}&=&\sup \{ r(p): p\in T_D(\Sigma_{\alpha_0}) \subset M^m\}.
\end{eqnarray}
Note that $r_{min}\le r_{D-}\le r_{D+}$ all depend on the
manifold while $r_0$ is uniquely determined by $\alpha_0$.
Since $r$ is a distance function, we have 
\be
r_0-D \le r_{D-}\le r_0 \le r_{D+} \le r_0+D \textrm{ and }0\le r_{min}
\ee
and the tubular neighborhood 
$T_D(\Sigma_{\alpha_0})\subset M^m$ embeds into
\be
r^{-1}(r_{D-}, r_{D+}) \subset {\mathbb{H}}^m \times {\mathbb{R}}.
\ee

\begin{figure}[h] 
   \centering
   \includegraphics[width=3in]{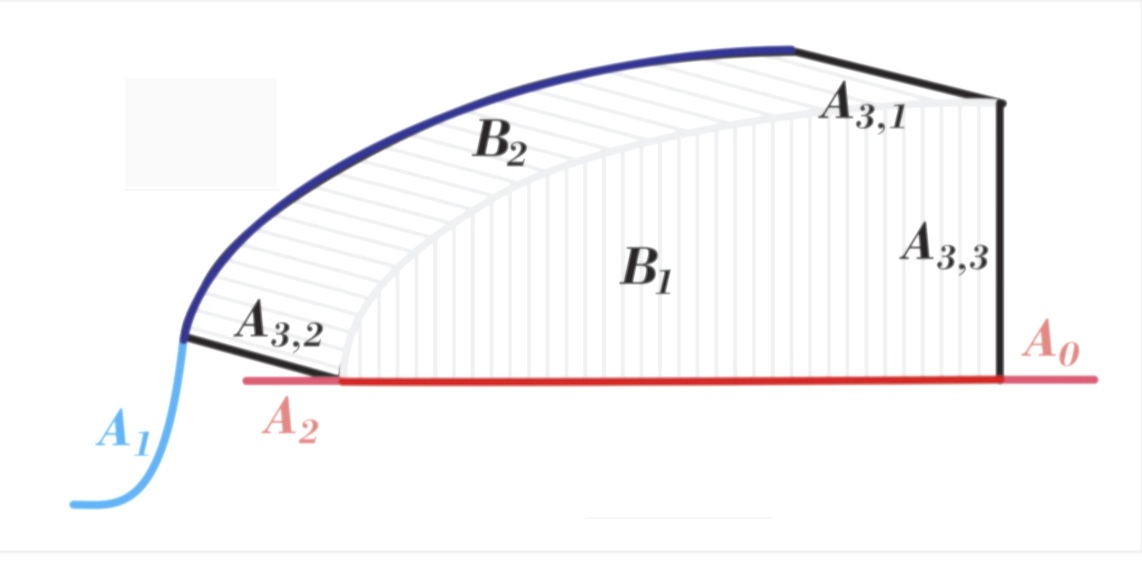} 
   \caption{
The horizontal line represents $T_D(\Sigma_{\alpha_0})\subset {\mathbb{H}}^m$ and the
curved line with a dip down on the left represents $T_D(\Sigma_{\alpha_0})\subset M^m$.  This is adapted from the corresponding figure in \cite{Lee-Sormani}.
      }
   \label{fig-SS-flat-filling}
\end{figure}

We will now explicitly define a filling between the two tubular neighborhood.  
All the regions we define starting
here are depicted in Figure~\ref{fig-SS-flat-filling}.
A part of our excess boundary will be formed by the region
\be
A_0=Ann_0(r_{D+}, r_0+D)\subset T_D(\Sigma_{\alpha_0})\subset {\mathbb{H}}^m
\ee
whose volume will
be estimated in Lemma~\ref{lem-switch-1}.  

In Lemma~\ref{lem-well} below,
we will choose a certain number $r_\epsilon'\in (0,r_0)$ and for 
\be \label{r-epsilon-prime}
r_\epsilon=\max\{r_\epsilon', r_{D-}\}
\ee
we cut off the
well 
\be\label{A-1}
A_1=r^{-1}(r_{D-},r_\epsilon)\subset T_D(\Sigma_{\alpha_0})\subset M^m.
\ee
Note that $A_1=\emptyset$ when
 $r_\epsilon'\le r_{D-}$ which occurs when our tubular neighborhood is not intersecting with a deep well.

In the case when $A_1$ is non-empty, we also cut off the corresponding annulus in hyperbolic space:
\be\label{A-2-1}
A_2=A_{2,1}=Ann_0(r_0-D,r_\epsilon) \subset T_D(\Sigma_{\alpha_0})\subset \mathbb{H}^m.
\ee
Note that when $r_0-D\le 0$, we have $A_2=B_0(r_\epsilon)$.
The volumes of $A_1$ and $A_2$ are uniformly controlled in Lemma~\ref{lem-well}.   

When $A_1$ is empty, we set
\be\label{A-2-2}
A_2=A_{2,2}=Ann_0(r_0-D,r_{D-}) \subset T_D(\Sigma_{\alpha_0})\subset 
{\mathbb{H}}^m.
\ee
The volume of $A_2$ is then bounded uniformly by Lemma~\ref{lem-switch-2}.

Recall that
Lemma~\ref{lem-graph} determines the embedding of \eqref{eqn-ann-M} up to
a vertical shift.  So we choose our Riemannian isometric embeddings of 
\begin{equation}\label{eqn-ann-M}
{r^{-1}(r_\epsilon, r_{D+})=
T_D(\Sigma_{\alpha_0})\setminus A_1
\subset M^m}
\end{equation}
and
\begin{equation}\label{eqn-ann-H}
{r^{-1}(r_\epsilon, r_{D+})= 
T_D(\Sigma_{\alpha_0})\setminus (A_0\cup A_2)
\subset {\mathbb{H}}^m}
\end{equation}
into 
$r^{-1}(r_\epsilon, r_{D+})\subset {\mathbb{H}}^m\times{\mathbb{R}}$
such that $\Sigma_{\alpha_\epsilon}\subset M^m$
and $\Sigma_{\alpha_\epsilon}\subset {\mathbb{H}}^m$ coincide. 
 
We define the region between the described isometric embeddings:
\be \label{B_1}
B_1=\{(x_1,\ldots,x_m,z): \,
z\in [0,F(r)],\, r\in (r_\epsilon, r_{D+})\}\subset
r^{-1}(r_\epsilon,r_{D+})\subset {\mathbb{H}}^{m}\times{\mathbb{R}}
\ee
where $F=F(r)$ is the embedding function.

Recall that the region $B_1$ is not a filling manifold: while the embedding of \eqref{eqn-ann-H} is distance preserving, the embedding of \eqref{eqn-ann-M} is not.
So we add a strip 
\be \label{B_2}
B_2=[0,S_M]\times r^{-1}[r_\epsilon, r_{D+}]\subset [0,S_M]\times M
\ee
of width $S_M$ determined by Lemma~\ref{lem-width}, then 
Theorem~\ref{thm-embed-const} 
gives us a distance preserving isometric embedding of \eqref{eqn-ann-M} and \eqref{eqn-ann-H}  
into $B_1 \cup B_2$.    

Applying Theorem~\ref{embed-const-2} we have the 
following proposition.  See Figure~\ref{fig-SS-flat-filling}.

\begin{prop} \label{main-estimates}
The intrinsic flat distance is bounded by the volumes:
\be \label{eqn-A_i-B_i}
d_{\mathcal{F}}(\,T_D(\Sigma_{\alpha_0})
\subset M^m\,
,\,T_D(\Sigma_{\alpha_0})\subset {\mathbb{H}}^m
\,)\le \vol_{m+1}(B)
+\vol_m(A)
\ee
where $B=B_1\cup B_2$ is the filling manifold
and 
\be
A=A_0+A_1+A_2+A_{3,1}+A_{3,2} + A_{3,3} 
\ee
is the excess boundary with
\be\label{A-3-1}
A_{3,1}=[0,S_M]\times r^{-1}\{r_{D+}\}\subset [0,S_M]\times M,
\ee
\be\label{A-3-2}
A_{3,2}=[0,S_M]\times r^{-1}\{r_\epsilon\}\subset [0,S_M]\times M,
\ee
\be\label{A-3-3}
A_{3,3}=r^{-1}\{r_{D+}\}\subset \partial B_1 \subset \mathbb{H}^{m}\times {\mathbb{R}}.
\ee
\end{prop}

The proof of Theorem~\ref{thm-main} will be completed by estimating the volumes of these regions.

\subsection{Cutting off the Deep Wells}\label{ss-cut-well}

In this section we carefully determine where to cut off a possibly deep well.

\begin{lem} \label{lem-well}
Given $\epsilon>0$, $D>0$, $\alpha_0>0$ and 
$M^m\in \RS$, let $\Sigma_{\alpha_0}\in M^m$ be
a symmetric sphere of area $\alpha_0$.
Set 
\be 
\alpha_\epsilon= \min\left\{\epsilon/ (16D), 
1/4, \omega_{m-1}\sinh^{m-1}(\epsilon/4)
, \alpha_0\right\}
\ee
and choose 
\be
r'_\epsilon=r(\Sigma_{\alpha_\epsilon})
=  \sinh^{-1}( (\alpha_\epsilon/\omega_{m-1})^{1/(m-1)} )
\ee
so that
\be
\alpha_\epsilon=\omega_{m-1} \sinh^{m-1} r'_\epsilon.
\ee
If the sets $A_1$ and $A_{2,1}$ are defined by (\ref{A-1}) and
(\ref{A-2-1}) respectively, we have
\be
\vol_m(B_0(r_\epsilon')\subset {\mathbb{H}}^m), \,\vol_m(A_1), \,\vol_m(A_{2,1}) 
\le \epsilon/8.
\ee
\end{lem}

\begin{proof}  
We have 
\be \label{r1}
\vol_m(B_0(r_\epsilon')\subset {\mathbb{H}}^m) \le r_\epsilon' \alpha_\epsilon
=
\alpha_{\epsilon} \sinh^{-1}( (\alpha_\epsilon/\omega_{m-1})^{1/(m-1)} )
< (1/4)(\epsilon/4)
= \epsilon/16.
\ee

Note that $A_{2,1}$ and $A_1$ are empty unless $r_\epsilon=r_\epsilon'$, 
so for the rest of the proof we may assume that $r_\epsilon=r_\epsilon'$. Then 
\be
\vol_m(A_{2,1}\subset B_0(r_\epsilon) )<\epsilon/16
\ee 
is a direct consequence of \eqref{r1}.

In order to estimate $\vol_m A_1$, we let $z_\epsilon=z(\Sigma_{\alpha_\epsilon})$ and 
$z_{D-}=\min\{z(p): \, p\in T_D(\Sigma_{\alpha_0})\subset M^m\}$.
Observe that $z_\epsilon-z_{D-}<D$
because we chose $\alpha_\epsilon< \alpha_0$
and areas are monotone in $\RS_m$ and $r^{-1}(r_{D-}, r_{D+})$
is a tubular neighborhood of radius $D$ about $\Sigma_0$.
Observe that the cylinder
\be
C^m=\partial B_0(r_\epsilon) \times [z_{D-}, z_\epsilon]
\ee
has volume 
\be
\vol_m(C^m)\le \alpha_\epsilon(z_\epsilon-z_{D-})\le
\alpha_\epsilon D\le \epsilon/16.
\ee
By Lemma~\ref{lem-F'} we have  $F'(z)\ge 0$ so
we can project the well, 
$A_1\subset M^m$,
radially outwards to $C^m$ and vertically downwards to $B_{0}(r_\epsilon)$
to estimate the volume:
\be
\vol_m(A_1)\le \vol_m(C^m) +\vol_m(B_0(r_\epsilon)) <\epsilon/8.
\ee
\end{proof}

\subsection{Initial Restriction on $\delta$ for Theorem~\ref{thm-main}}

Recall that in Theorem~\ref{thm-main} we are given 
any $\epsilon>0$, $D>0$, $A_0>0$ and $m\in \N$,
and must choose $\delta=\delta(\epsilon, D, A_0, m)>0$
depending only on these parameters such that 
if $M^m\in\RS_m$ has mass, 
$\mathrm{m}_{\mathrm{AH}}(M)<\delta$ we will have our
estimate on the intrinsic flat distance.    Here we will
start choosing a number $\delta>0$ which will give us strong enough
control to bound $|F'|$ which will then 
bound the embedding constant and the width of the strip. 


\begin{lem} \label{lem-Q}
Given fixed $r_\epsilon>0$ 
and $m\in \N$, 
choose 
\be\label{eq-r-epsilon}
\delta< \delta(r_\epsilon)=\tfrac{1}{2} \sinh^{m-2} r_{\epsilon} \cosh^2 r_\epsilon.
\ee
If $M^m\in \RS_m$ has $\mathrm{m}_{\mathrm{AH}}<\delta$ 
then $M^m$ has an isometric embedding into
\be
\{z=F(r)\}\subset {\mathbb{H}}^{m}\times{\mathbb{R}}
\ee
where $F:[r_{min},\infty) \to \R$ is
an increasing function, and
\be \label{rmin-delta}
r_{min}\le r_\delta < r_\epsilon,
\ee
where $r_\delta$ is the unique positive root of the equation 
\be\label{eq-r-delta}
2\delta = \sinh^{m-2}r_\delta \cosh^2 r_\delta.
\ee
Furthermore, we have
\be\label{eq-Q}  
|F'(r)| \le Q(\delta,r_\epsilon)\qquad \forall r\ge 
r_{\epsilon},
\ee
where  
\be\label{defQ}
Q(\delta,r):=\sqrt{2\delta/(\sinh^{m-2} r\cosh^2 r-2\delta)}
\ee
is such that
\be
\lim_{\delta\to 0}Q(\delta,r_\epsilon)=0
\ee 
when $r_\epsilon$ is fixed.
\end{lem}

\begin{proof}
Lemma~\ref{lem-graph} provides the Riemannian
isometric embedding. By Lemma~\ref{lem-rmin} we have $r_{min} \leq  r_\infty$, where $r_\infty$ satisfies
\[
\sinh^{m-2}r_\infty \cosh^2 r_\infty = 2 m_{\mathrm{AH}} < 2 \delta < 2\delta (r_\epsilon).
\]
This together with \eqref{eq-r-epsilon} and \eqref{eq-r-delta}
provides (\ref{rmin-delta}).
Lemma~\ref{lem-F'} and the fact that
$\mathrm{m}_{\mathrm{AH}}(M)< \delta$ then implies that
\be
|F'(r)| \le Q(\delta,r) := \sqrt{2\delta/(\sinh^{m-2}r\cosh^2 r - 2\delta)} 
\qquad \forall r\ge
r_\delta.
\ee
By our choice of $\delta$ we have $r_\epsilon > r_\delta$,
so we get (\ref{eq-Q}) by applying the fact that
$Q(\delta,r)$ decreases in $r$. 
\end{proof}

Note that the results of Section \ref{ss-cut-well} allow us to control all regions with $r<r_\epsilon$. Lemma \ref{lem-Q} will enable us to controll the remaining regions of interest
by taking $\delta$ small enough.

\subsection{Estimating the Regions in Hyperbolic space}

In this section we estimate the volumes of the regions 
$A_0$ and $A_{2}$ proving Lemma~\ref{lem-switch-1}
and Lemma~\ref{lem-switch-2}.  

\begin{lem}\label{lem-switch-1}
Given $D>0$, $\alpha_0>0$, $m\in \N$, we choose
$\delta>0$ as in Lemma~\ref{lem-Q}.
If $M^m\in \RS_m$ has $\mathrm{m}_{\mathrm{AH}}<\delta$ 
then
\be
\vol_m(A_0)\le D Q(\delta,r_0) \omega_{m-1}\sinh^{m-1}(r_0+D)
\ee
where $Q(\delta,r_0)$ is defined in (\ref{defQ}).
\end{lem}

\begin{proof}  
By our choice of $\delta$ as in Lemma~\ref{lem-Q}, we have
\be
|F'(r)| \le Q(\delta, r_0) \qquad \forall r\ge r_0.
\ee
As a consequence, by the formula for arclength we obtain
\begin{eqnarray}
r_0+D-r_{D+}&=&r_0-r_{D+}+\int_{r_0}^{r_{D+}} \sqrt{1+F'(r)^2}\, dr\\
&\le & r_0-r_{D+}+(r_{D+}-r_0) (1+ Q(\delta,r_0))\\
&\le & (r_{D+}-r_0) Q(\delta,r_0)\le DQ(\delta,r_0).\\
\end{eqnarray}
Thus 
\begin{eqnarray}
\vol(A_0)&\le & \vol(r^{-1}(r_{D+}, r_0+D)\subset {\mathbb{H}}^m)\\
&\le & \left(r_{0}+D-r_{D+}\right) \omega_{m-1}\sinh^{m-1}(r_0+D)\\
&\le & D Q(\delta,r_0) \omega_{m-1}\sinh^{m-1}(r_0+D)
\end{eqnarray}
and the lemma follows.
\end{proof}

Recall that the region $A_{2,2}$ in (\ref{A-2-2}) is only
defined when $r'_\epsilon \le r_{D-}$, in which case $r_\epsilon = r_{D-}$.  This condition is assumed in the
following lemma estimating the volume of $A_{2,2}$.

\begin{lem}\label{lem-switch-2}   
Given $r_\epsilon>0$, $D>0$, $\alpha_0>0$, $m\in \N$,
we choose $\delta$ as in Lemma~\ref{lem-Q}.
If $M^m\in \RS_m$ has $\mathrm{m}_{\mathrm{AH}}<\delta$ then 
the region $A_{2}$ defined in (\ref{A-2-2}) satisfies
\be
\vol_m(A_{2}) \le D Q(\delta,r_\epsilon) \omega_{m-1}\sinh^{m-1}(r_0).
\ee
\end{lem}

\begin{proof}
By our choice of $\delta$ we know that
\be
|F'(r)| \le Q(\delta, r_\epsilon) \qquad \forall r\ge r_{D-}= r_\epsilon,
\ee
so by the formula for arclength we have
\begin{eqnarray}
r_{D-}-(r_0-D)&=&r_{D-}-r_{0}+\int_{r_{D-}}^{r_{0}} \sqrt{1+F'(r)^2}\, dr\\
&\le & r_{D-}-r_0+(r_{0}-r_{D-}) (1+ Q(\delta,r_\epsilon))\\
&\le & (r_{0}-r_{D-}) Q(\delta,r_\epsilon)\le DQ(\delta,r_\epsilon).
\end{eqnarray}
Thus 
\begin{eqnarray}
\vol(A_{2})&= & \vol(r^{-1}(r_0-D, r_{D-})\subset {\mathbb{H}}^m)\\
&\le & \left(r_{D-}-(r_0-D)\right) \omega_{m-1}\sinh^{m-1}(r_{D-}) \\
&\le & D Q(\delta,r_\epsilon) \omega_{m-1}\sinh^{m-1}(r_0).
\end{eqnarray}
\end{proof}

\subsection{Controlling the Width of the Strip}
\label{ss-choose}

\begin{lem}\label{lem-width}
Given $r_\epsilon>0$, $D>0$, $\alpha_0>0$, $m\in \N$,
we choose $\delta$ as in Lemma~\ref{lem-Q}.
If $M^m\in \RS_m$ and
$\mathrm{m}_{\mathrm{AH}}(M^m) < \delta$,
then the region $r^{-1}(r_\epsilon,r_{D+})$
isometrically embeds into the filling
manifold $B_1\cup B_2$ of (\ref{B_1})
and (\ref{B_2})  
where
\be\label{strip-1}
S_M=S(\delta,r_\epsilon, D, r_0)
= \sqrt{C (2D +\pi \sinh r_0 + C)}
\ee
with
\be \label{emb-1}
C=C(\delta,r_\epsilon, D, r_0)= (4D+2\pi \sinh r_0) Q(\delta, r_\epsilon),
\ee
where $Q(\delta, r_\epsilon)$ is defined as in Lemma~\ref{lem-Q}.
\end{lem}

\begin{proof}
We begin by applying Theorem~\ref{thm-embed-const} to
the isometric embedding of
\be
r^{-1}(r_\epsilon, r_{D+})\subset M^m
\ee
into
\be
N^{m+1}=W \times \R \subset {\mathbb{H}}^{m}\times \R
\ee
where $W=Ann_0(r_\epsilon, r_{D+})\subset {\mathbb{H}}^m$.

Since $r^{-1}(r_\epsilon, r_{D+})\subset T_D(\Sigma_{\alpha_0})$,
we have
\be
\diam(W)\le\diam(r^{-1}(r_\epsilon, r_{D+})) \le 2D + \diam(\Sigma_{\alpha_0})=2D+\pi \sinh r_0.
\ee
By Lemma~\ref{lem-Q} we have a bound, $|F'(r)|\le Q(\delta, r_\epsilon)$, for any $r\geq r_\epsilon$. Then 
Theorem~\ref{thm-Z} gives us the 
embedding constant $C_M\le C$
as in (\ref{emb-1}).  By Theorem~\ref{thm-embed-const}, the strip must have width
$S_M$ as in (\ref{strip-1}).
\end{proof}

\subsection{Volume Estimates and the Proof of Theorem~\ref{thm-main}}
\label{ss-vol}

In this section we complete the proof of Theorem~\ref{thm-main}
by estimating the volumes of the filling manifold $B_1\cup B_2$ and the excess boundary $A$ and applying (\ref{eqn-A_i-B_i}) as required in Proposition~\ref{main-estimates}.

\begin{proof}
Given any $\epsilon>0$, $D>0$, $\alpha_0>0$, $m\in \N$  we set
$r_0>0$ so that 
\be
\alpha_0=\omega_{m-1} \sinh^{m-1} r_0
\ee 
and choose 
\be
r_\epsilon'=r_\epsilon'(\epsilon, D, \alpha_0, m)>0
\ee 
exactly as in Lemma~\ref{lem-well}. Further, we define 
\be
r_\epsilon=\max\{r_\epsilon', r_{D-}\}
\ee
as in (\ref{r-epsilon-prime})  and choose 
\be\label{choose-delta-1}
\delta<\delta(r_\epsilon)
\ee 
as in Lemma~\ref{lem-Q}. This will be subsequently refined to obtain the value of $\delta$ which suffices for Theorem \ref{thm-main} to hold. From now on it will be assumed that $M^m\in\RS_m$ has mass 
$\mathrm{m}_{\mathrm{AH}}<\delta$.

When $r'_\epsilon \geq r_{D-}$ we apply Lemma~\ref{lem-well},
(\ref{A-1}) and
(\ref{A-2-1}) to see that 
\be\label{201}
\vol_m(A_1)+\vol_m(A_2)\le \epsilon/8+\epsilon/8=\epsilon/4.
\ee
If $r'_\epsilon< r_{D-}$ then
$A_1=\emptyset$ and by Lemma~\ref{lem-switch-2}, we
obtain the same estimate as long as
$\delta>0$ is chosen small enough so that
\be\label{choose-delta-2}
D Q(\delta, r_\epsilon)
\omega_{m-1} \sinh^{m-1} (r_0 + D) < \epsilon/8
\ee
holds. This second restriction on $\delta$ also suffices to
obtain
\be\label{203}
\vol_m(A_0)<\epsilon/8,
\ee
see Lemma \ref{lem-switch-1}.

By Lemma~\ref{lem-graph} we have an
isometric embedding of $r^{-1}(r_\epsilon,r_{D+})$
into the graph, $\{z=F(r)\}\subset {\mathbb{H}}^{m}\times{\mathbb{R}}$, and may
define $B_1$ as in (\ref{B_1}).  We then have
\begin{eqnarray}
\vol_{m+1}(B_1)
&=& \int_{r_\epsilon}^{r_{D+}} (F(r)-F(r_\epsilon)) \omega_{m-1} \sinh^{m-1} r \, dr\\
&\le & (r_{D+}-r_\epsilon) \omega_{m-1} \sinh^{m-1} r_{D+}
(F(r_{D+})-F(r_\epsilon))\\
&\le & 2D \omega_{m-1} \sinh^{m-1} (r_0+D)
\int_{r_\epsilon}^{r_{D+}} F'(r)\, dr < \epsilon/8
\end{eqnarray}
as long as $\delta$ is chosen small enough so that
\be \label{choose-delta-3}
4D^2 \omega_{m-1}\sinh^{m-1} (r_0+D)Q(r_\epsilon, \delta) < \epsilon/8
\ee
holds. 

Next we apply Lemma~\ref{lem-width} to
create region $B_2$ as in (\ref{B_2}). Again, by choosing $\delta$ to be sufficiently small, we can ensure that
\begin{eqnarray} 
\quad \quad\vol_{m+1}(B_2)
&=&  S_M \vol(r^{-1}(r_\epsilon, r_{D+})\subset M^m) \\
&=& 
S_M \int_{r_\epsilon}^{r_{D+}} \sqrt{1+F'(r)^2}\, \omega_{m-1}\sinh^{m-1} r \, dr\\
&=& 
S_M \int_{r_\epsilon}^{r_{D+}} \left(1+F'(r)\right)\, \omega_{m-1}\sinh^{m-1} r \, dr\\
&\leq&  S(\delta,r_\epsilon, D, r_0) 2D \omega_{m-1}
\sinh^{m-1} (r_0+D) (1+Q(\delta, r_\epsilon))<\epsilon/8
\end{eqnarray}
 since 
\be 
\lim_{\delta \to 0} S(\delta,r_\epsilon, D, r_0) = 0
\ee
by  \eqref{strip-1}, \eqref{emb-1}, and \eqref{defQ}. 

Further, by (\ref{A-3-1}) and (\ref{A-3-2})  we have
\begin{eqnarray}
\quad \quad  \vol_m(A_{3,1})&=& S_M \,\omega_{m-1}\sinh^{m-1} r_{D+}
\,\,\,\le\,\,\, S_M
\omega_{m-1}\sinh^{m-1}(r_0+D) < \epsilon/12\\
\vol_m(A_{3,2})&=& S_M \,\omega_{m-1}\sinh^{m-1} r_\epsilon \le S_M \omega_{m-1}\sinh^{m-1}r_0 <\epsilon/12
\end{eqnarray}
as long as $\delta$ is chosen small enough so that
\be \label{choose-delta-5}
S(\delta, r_\epsilon, D, r_0)\omega_{m-1}\sinh^{m-1}(r_0+D)< \epsilon/12
\ee
holds. Finally, by (\ref{A-3-3}) and the mean value theorem we have
\begin{eqnarray}
\quad \vol_m(A_{3,3})&=& \omega_{m-1}\sinh^{m-1}(r_{D+}) (F(r_{D+})-F(r_\epsilon))\\
&\le& 2D \omega_{m-1}\sinh^{m-1}(r_0+D) Q(\delta, r_\epsilon)
< \epsilon/12
\end{eqnarray}
as long as $\delta$ is chosen small enough so that
\be \label{choose-delta-6}
\quad 2 D \omega_{m-1}\sinh^{m-1}(r_0+D) Q(\delta, r_\epsilon)
< \epsilon/12.
\ee
The estimate on the intrinsic flat distance then follows from (\ref{eqn-A_i-B_i}),
adding all the estimated volumes.

Applying (\ref{201})-(\ref{203}) again,
we obtain the volume estimate because
\[
\begin{split}
&\vol(\,T_D(\Sigma_0)\subset M^m\,)\\
& \qquad =\,\,\,\vol(A_1) \,\,+ \,\,\vol(\,r^{-1}(r_\epsilon, r_{D+})\subset M^m\,)\\
&\qquad =\,\,\,\vol(A_1) \,\,+ \int_{r_\epsilon}^{r_{D+}} \sqrt{1+F'(r)^2\,} \omega_{m-1}
\sinh^{m-1} r \,dr\\
&\qquad \leq \,\,\,\vol(A_1) \,\, + \sqrt{1+Q(\delta, r_\epsilon)^2\,}
\int_{r_\epsilon}^{r_{D+}} \omega_{m-1}\sinh^{m-1}r \,dr\\
&\qquad =\,\,\,\vol(A_1) \,\, + \sqrt{1+Q(\delta, r_\epsilon)^2\,}\vol(r^{-1}(r_\epsilon, r_{D+})\subset {\mathbb{H}}^m)\\
&\qquad =\,\,\,\vol(A_1) \,\, + \sqrt{1+Q(\delta, r_\epsilon)^2\,}
\left(\vol(T_D(\Sigma_{\alpha_0})\subset {\mathbb{H}}^m)-\vol(A_0)-\vol(A_2) \right)
\end{split}
\]
and the fact that $Q(\delta, r_\epsilon)$ can be taken small as $\delta\to 0$.
\end{proof}

\bibliographystyle{plain}
\bibliography{2017}

\end{document}